\documentclass{amsart}
\usepackage{amssymb,latexsym}

\input xypic
\usepackage[bookmarks=true]{hyperref}       

\usepackage{amssymb,latexsym,amsmath,amscd}
\usepackage{xspace}
\usepackage{enumerate}
\usepackage{graphicx}
\usepackage{vmargin}
\usepackage{todonotes}

\reversemarginpar

\theoremstyle{plain}
\newtheorem{theorem}{Theorem}[section]
\newtheorem*{theorem*}{Theorem}

\newtheorem{lemma}[theorem]{Lemma}

\theoremstyle{definition}

\newtheorem{remark}[theorem]{Remark}

\newcommand{\enm}[1]{\ensuremath{#1}}          %

\newcommand{\cal}[1]{\mathcal{#1}}

\newcommand{\CC}{\enm{\mathbb{C}}}

\newcommand{\RR}{\enm{\mathbb{R}}}
\newcommand{\QQ}{\enm{\mathbb{Q}}}

\newcommand{\PP}{\enm{\mathbb{P}}}

\newcommand{\Bb}{\enm{\cal{B}}}

\newcommand{\Ii}{\enm{\cal{I}}}

\newcommand{\Ss}{\enm{\cal{S}}}

\renewcommand{\phi}{\varphi}
\renewcommand{\theta}{\vartheta}
\renewcommand{\epsilon}{\varepsilon}

\begin{document}

\title[symmetric tensors]
{An effective criterion for the additive decompositions of forms}
\author{Edoardo Ballico}
\address{Dept. of Mathematics\\
 University of Trento\\
38123 Povo (TN), Italy}
\email{ballico@science.unitn.it}
\thanks{The author was partially supported by MIUR and GNSAGA of INdAM (Italy).}
\subjclass[2010]{14N05; 15A69; 15A72; 14N20}
\keywords{symmetric tensor rank; additive decomposition of polynomials; Waring decomposition; identifiability}

\begin{abstract}
We give an effective criterion for the identifiability of additive decompositions of homogeneous forms of degree $d$ in a fixed
number of variables. Asymptotically for large $d$ it has the same order of the  Kruskal's criterion  adapted to symmetric tensors given by
 L. Chiantini, G. Ottaviani and N. Vannieuwenhoven. We give a new case of indentifiability for $d=4$.
\end{abstract}

\maketitle

\section{Introduction}
Let $\mathbb {C}[z_0,\dots ,z_n]_d$ denote the complex vector space of all homogeneous degree $d$ polynomials in the variables
$z_0,\dots ,z_n$. An \emph{additive decomposition} (or a Waring decomposition) of  a form $f\in \CC[z_0,\dots ,z_n]_d\setminus
\{0\}$ is a finite sum
\begin{equation}\label{eqi1} f = \sum  \ell _i^d
\end{equation}
with each $\ell _i\in \CC[z_0,\dots ,z_n]_1$. The minimal number $R (f)$ of addenda
in an additive decomposition of $f$ is called the rank of $f$. Degree $d$ forms in the variables $z_0,\dots ,z_n$ correspond
to symmetric tensors of format $(n+1)\times \cdots \times (n+1)$ ($d$ times), i.e.  to symmetric elements of
$(\CC^{n+1})^{\otimes d}$. An additive decomposition (\ref{eqi1}) of $f$ is said to be \emph{minimal} if there are no $c_i\in
\CC$ with at least one
$c_i=0$ such that $f =\sum _ic_i\ell _i^d$. See \cite{l} for a long list of possible applications and the language needed.
Obviously it is interesting to know when a minimal decomposition of
$f$ has only $R (f)$ addenda i.e. knowing it we also know $R (f)$. More important is to know that there are no other additive decompositions of $f$ with $R (f)$ addenda
(obviously up to a permutation of the addenda $\ell _i^d$). 

In \cite{cov2} L. Chiantini, G. Ottaviani and N. Vannieuwenhoven stressed the
importance (even for arbitrary tensors) of effective criteria for the identifiability and gave a long list of practical
applications (with explicit examples even in Chemistry). We add to the list (at least in our hope) the tensor networks
(\cite{bberg, bbl, o}, at least for tensors without symmetries. For the case of bivariate forms, see \cite{bc} (but for
bivarariate forms the identifiability of a form only depends from its rank and, for generic bivariate forms, by the parity of
$d$ by a theorem of Sylvester's (\cite[Theorem 1.5.3 (ii)]{ik}). 

L. Chiantini, G. Ottaviani and N. Vannieuwenhoven stressed the importance of the true effectivity of the criterion to be
tested (as the famous Kruskal's criterion for the tensor decomposition (\cite{kr}). They reshaped the Kruskal's criterion to the case of
additive decompositions (\cite[Theorem 4.6 and Proposition 4.8]{cov2}) and proved that is effective (for $d\ge 5$) for ranks at
most
$\sim n^{\lfloor (d-1)/2\rfloor}$. The upper bound to apply our criterion has the same asymptotic order when $d\gg 0$, but we
hope that it is easy and efficient. Then in \cite{acv} E. Angelini, L. Chiantini and N. Vannieuwenhoven considered the case
$d=4$ and added the analysis of one more rank. Among the huge number of papers considering mostly ``generic '' identifiability
we also mention
\cite{an, abc, agmo, bco, chio, cov0, cov1}.

To state our results we need the following geometric language for instance fully explained, e.g., in \cite{cov2, l}.

Set  $\PP^n:= \PP
\CC[z_0,\dots ,z_n]_1$. Thus points of the $n$-dimensional complex space $p$ corresponds to non-zero linear forms, up  to a
non-zero multiplicative constant. Set $r:= \binom{n+d}{n}-1$. Thus $\PP^r:= \PP\CC[z_0,\dots ,z_n]_d$ is an $r$-dimensional
projective space. Let $\nu _d: \PP^n\to \PP^r$ denote the order $d$ Veronese embedding, i.e. the map defined by the formula
$[\ell]\mapsto [\ell ^d]$. An additive decomposition (\ref{eqi1}) with $k$ non-proportional non-zero terms corresponds to a
subset $S\subset \PP^n$ such that $|S|=k$ and $[f] \in \langle \nu _d(S)\rangle$, where $\langle \ \ \rangle$ denote the
linear span. This decomposition is called \emph{minimal} and we say that the set $\nu _d(S)$ \emph{minimally spans} $[f]$ if
$[f]\in \langle \nu _d(S)\rangle$ and $[f]\notin \langle \nu _d(S')\rangle$ for each $S'\subsetneq S$. For any integer $t\ge 0$
each $p\in \PP^n$ gives a linear condition to the vector space $\CC[z_0,\dots ,z_n]_t$ taking $p_1=(z_0,\dots ,z_n)\in
\CC^{n+1}$ with $[p_1] =p$ and evaluating each $f\in \CC[z_0,\dots ,z_n]_t$ at $p_1$. When we do this evaluation for all points
of a finite set $S\subset \PP^n$ we get
$|S|$ linear equations and the rank of the corresponding matrix does not depend from the choice of the representatives of the
points of $S$.

We prove the following result.

\begin{theorem}\label{i1}
Fix $q\in \PP^r$ and take a finite set $S\subset \PP^n$ such that $\nu _d(S)$ minimally spans $q$. 

\quad (a) If $|S| \le \binom{n+\lfloor d/2\rfloor}{n}$ and $S$ gives $|S|$ independent conditions to the complex vector space $\CC[z_0,\dots ,z_n]_{\lfloor d/2\rfloor}$, then $r_X(q) =|S|$.

\quad (b) If $|S| \le \binom{n+\lfloor d/2\rfloor -1}{n}$ and $S$ gives $|S|$ independent conditions to the complex vector space $\CC[z_0,\dots ,z_n]_{\lfloor d/2\rfloor -1}$, then $S$ is the unique set evincing the rank of $q$.
\end{theorem}

In Remark \ref{a0} we explain why Theorem \ref{i1} effectively determines the rank of $q$ (and in the set-up of (b) the uniqueness statement often  called ``uniqueness of additive decomposition ''
for homogeneous polynomials or for symmetric tensors). Indeed, to check that $S$ satisfies the assumptions of part (a) (resp. part (b)) of Theorem \ref{i1} it is sufficient to check that a certain matrix with $|S|$ rows and $\binom{n+\lfloor d/2\rfloor}{n}$ (resp $\binom{n+\lfloor d/2\rfloor -1}{n}$) columns has rank $|S|$. This matrix has rank $|S|$ if $S$ is sufficiently general, but the test is effective for a specific set $S$.

See \cite{bbc} and \cite{bbcg} for results similar to Theorem \ref{i1} for tensors; roughly speaking \cite[Corollary 3.10, Remark 3.11 and their proof]{bbcg} is equivalent to part (a) of Theorem \ref{i1}. Part (a) of Theorem \ref{i1} is good, but one could hope
to get part (b) when
$|S| <
\binom{n+\lfloor d/2\rfloor}{n}$, adding some other easily testable assumptions on $S$. We prove the following strong result
(an essential step for the proof of part (b) of Theorem
\ref{i1}). To state it we recall the following notation: for any finite set $E\subset \PP^n$ and any $t\in \mathbb {N}$ let
$H^0(\Ii _E(t))$ denote the set of all $f\in \CC[z_0,\dots ,z_n]_t$ such that $f({p})=0$ for all $p\in E$. The set $H^0(\Ii
_E(t))$ is a vector space of dimension at least $\binom{n+t}{n} -|E|$.

\begin{theorem}\label{a3}
Fix $q\in \PP^r$ and take a finite set $S\subset \PP^n$ such that $\nu _d(S)$ minimally spans $q$. Assume $|S| <
\binom{n+\lfloor d/2\rfloor}{n}$ and that $S$ gives $|S|$ gives independent conditions to $\CC[z_0,\dots ,z_n]_{\lfloor
d/2\rfloor}$. Take any $A\subset \PP^n$ such that $|A| =|S|$ and $A$ induces an additive decomposition of $f$.  Let $H^0(\Ii
_A(\lfloor d/2\rfloor ))$ (resp. $H^0(\Ii _S(\lfloor d/2\rfloor ))$) denote the linear subspace of $\CC[z_0,\dots
,z_n]_{\lfloor d/2\rfloor}$ formed by the polynomials vanishing at all $p\in A$ (resp. $p\in S$). Then $H^0(\Ii _A(\lfloor
d/2\rfloor )) =H^0(\Ii _S(\lfloor d/2\rfloor ))$.
\end{theorem}

Theorem \ref{a3} does not assure that $S$ is the only set evincing the rank of $q$, i.e. the uniqueness of the addenda in an
additive decomposition of $f$ with $R(f)$ terms, but it shows where the other sets $A$ giving potential additive decomposition with $R(f)$ addenda may be located. The results in
\cite{acv} (in particular \cite[Theorem 6.2 and 6.3, Proposition 6.4]{acv}) for
$d=4$ show that non-uniqueness does occur if and only if the base locus of $|\Ii _S(2)|$ allows the existence of $A$.

In the last section we take $d=4$. E. Angelini, L. Chiantini and N.  Vannieuwenhoven consider the case $d=4$ and $|S|=2n+1$
with an additional geometric property (linearly general position or GLP for short; section \ref{s4} for its definition). For
$d+4$ and $|S|=2n+1$ they classified the set
$S$ in GLP for which identifiability holds. In section \ref{S4} using Theorem \ref{a3} we classify another family of sets $S$
with
$|S|=2n+1$ and for which identifiability holds (Theorem \ref{i0}).

\begin{remark}
The results used to prove Theorem \ref{i1} (and summarized in Remarks \ref{a1} and \ref{a2}) works verbatim for a
zero-dimensional schemes $A\subset \PP^n$. Under the assumption of part (a) of Theorem \ref{i1} the cactus rank of $q$
(see \cite{bbm, br, rs} for its definition and its uses) is $|S|$.  Under the assumptions of part (b) of Theorem \ref{i1} $S$ is the only zero-dimensional subscheme of $\PP^n$
evincing the cactus rank of $q$. However for our proofs it is important that $S$ (i.e. the scheme to be tested) is finite set,
not a zero-dimensional scheme. 
\end{remark}

\begin{remark}
The interested reader may check that the proof works with no modification if instead of $\CC$ we take a any algebraically
closed field containing $\QQ$. Since it uses only linear systems, it works over any field $K\supseteq \QQ$ if as an additive
decomposition of $f\in K[z_0,\dots ,z_n]_d$ we take an expression $\sum c_i\ell _i^d$ with $c_i\in K$ and $\ell _i\in
K[z_0,\dots ,z_n]_1$. Thus for the real field $\RR$ when $d$ is odd we may take the usual definition (\ref{eqi1}) of additive decomposition,
while if $d$ is even we allow $c_i\in \{-1,1\}$. Theorem \ref{i1} applied to $\CC$ says that $|S|$ is the complex rank of
$q$, too, and in set-up of part (b) uniqueness holds even if we allow complex decompositions.
\end{remark}

\begin{remark}
In the proofs of our results we use nothing about the form $f$ or the point $q=[f]\in \PP^r$. All our assumptions are on
the set
$S$ and they apply to all $q\in \langle \nu _d(S)\rangle$ minimally spanned by $\nu _d(S)$. In all our results the set $\nu
_d(S)$ is linearly independent and hence the set of all $q\in \PP^r$ minimally spanned by $\nu _d(S)$ is the complement in the
$(|S|-1)$-dimensional linear space $\langle \nu _d(S)\rangle$ of $|S|$ codimension $1$ linear subspaces. To test that $\nu
_d(S)$ minimally spans $q$ it is sufficient to check the rank of a matrix with $|S|$ rows and $\binom{n+d}{n}$ columns. To
the best of our knowledge this check (or a very similar one) must be done for all criteria of effectivity for forms
(\cite{acv}).
\end{remark}

\section{The proofs of Theorems \ref{i1} and \ref{a3}}
Fix $q\in \PP^r = \PP ^r = \PP
\CC[z_0,\dots ,z_n]_d$. Recall that a finite subset $E\subset \nu _d(\PP^n)$ minimally spans $q$ if $q\in 
\langle E\rangle$ and $q\notin \langle E'\rangle$ for any $E'\subsetneq E$. Note that if $E$ minimally spans a point of $\PP^r$, then
it is linearly independent, i.e. $\dim \langle E\rangle = |E|-1$. If $E = \nu _d(A)$ for some $A\subset \PP^n$, $E$ is
linearly independent if and only if $A$ induces $|A|$ independent conditions to $\CC[z_0,\dots ,z_n]_d$.

\begin{remark}\label{a0}
Fix an integer $t\ge 0$ and a finite subset $A$ of $\PP^n$. We write $h^1(\Ii _A(t))$ for the
difference between $|A|$ and the number of independent conditions that $A$ imposes to the $\binom{n+t}{n}$-dimensional
vector space $\CC[z_0,\dots ,z_n]_t$. Taking as a basis of $\CC[z_0,\dots ,z_n]_t$ all degree $t$ monomials in $z_0,\dots ,z_n$
to compute the non-negative integer $h^1(\Ii _A(t))$ we only need to compute the rank of the matrix with $|A|$ rows and
$\binom{n+t}{n}$ columns. Indeed, the linear span of $\nu _d(A)$ is computed by the linear system obtained evaluating the
polynomial $\sum c_\alpha z^\alpha$ with $c_\alpha$ variables at each $p\in A$.
\end{remark}

Fix $f\in \CC[z_0,\dots ,z_n]_d \setminus \{0\}$ and let $q =[f]\in \PP^r =\PP \CC[z_0,\dots ,z_n]_d$, $r =\binom{n+d}{n}-1$,
be the point associated to $f$. Take $S\subset \PP^n$ such that $\nu _d(S)$ minimally spans $q$. Fix any $A\subset \PP^n$
evincing the rank of $f$. We have $|A| \le |S|$. Set $Z:= A\cup B$. $Z$ is a finite subset of $\PP^n$ and $|Z| \le |A|+|S|$.
To prove part (a) of Theorem \ref{i1} we need to prove that $|A|=|S|$. To prove part
(b) we need to prove that $A=S$. In the proof of part (a) we have $A\ne S$, because $|A|<|S|$. To prove part (b) of the
theorem it is sufficient to get a contradiction from the assumption $A\ne S$. Thus from now on we assume $A\ne S$. Since
$A\ne S$ and both $\nu _d(A)$ and $\nu _d(S)$ minimally span $q$, we have $h^1(\Ii _Z(d)) \ne 0$ (\cite[Lemma 1]{bb}; see Remark \ref{a1} for more details.

\begin{remark}\label{a1}
Fix $q\in \PP^r = \PP ^r = \PP \CC[z_0,\dots ,z_n]_d$. For all linear subspaces $U, W\subseteq \PP^r$ the Grassmann's formula
says that 
$$\dim U\cap W + \dim U+W = \dim U + \dim W$$with the convention $\dim \emptyset = -1$. Fix $q\in \PP^r = \PP ^r = \PP
\CC[z_0,\dots ,z_n]_d$ and assume the existence of finite sets $A, B\subset \PP^n$ such that $\nu _d(A)$ and $\nu _d(B)$
minimally span $q$. Since $\nu _d(A)$ (resp. $\nu _d(B)$) minimally spans $q$, we have $\dim
\langle
\nu _d(A)\rangle =|A|-1$ (resp. $\dim \langle \nu _d(A)\rangle =|A|-1$).  Since
$A\ne B$, we have
$A\cap B\subsetneq A$ and
$A\cap B\subsetneq B$. Since
$q\in
\langle
\nu _d(A)\rangle
\cap
\langle
\nu _d(B)\rangle$ and $q\notin \langle \nu _d(A\cap B)\rangle$, we have $\langle \nu _d(A)\cap \rangle \nu _d(B)\rangle
\supsetneq \langle \nu _d(A\cap B)\rangle$. Since $\nu _d(A)$ and $\nu _d(B)$ are linearly
independent and  $\langle \nu _d(A)\cap \rangle \nu _d(B)\rangle
\supsetneq \langle \nu _d(A\cap B)\rangle$, the Grassmann's formula gives that $\nu _d(A\cup B)$ is not linearly independent,
i.e.
$A\cup B$ does not impose $|A\cup B|$ independent conditions to $\CC[z_0,\dots ,z_n]_d$ (\cite[Lemma 1]{bb}).
\end{remark}

\begin{remark}\label{a2}
We explain the particular case of \cite[Lemma 5.1]{bb2} we need. Fix $q\in \PP^r = \PP ^r = \PP
\CC[z_0,\dots ,z_n]_d$ and take finite sets $A, B\subset \PP^n$ such that $\nu _d(A)$ and $\nu _d(B)$ minimally spans $q$.
In particular both $A$ and $B$ are linearly independent. Set $Z:= A\cup B$. We assume that
$Z\setminus Z\cap G$ gives $|Z\setminus Z\cap G|$ independent conditions to $\CC[z_0,\dots ,z_n]_{d-t}$, i.e. we assume
$h^1(\PP^n,\Ii _{Z\setminus Z\cap G}(d-t)) =0$. We want to prove that $A\setminus A\cap G=B\setminus B\cap G$ (\cite[Lemma
5.1]{bb2}). Since this is obvious if $Z\subset G$, we may assume $Z\ne Z\cap G$, say $A\ne A\cap
G$; just to fix the notation we also assume $A\cap G\ne \emptyset$. Set
$E:= A\cap B
\setminus A\cap B\cap G$ and write $A\setminus A\cap G =A\cup A_1$ and $B\setminus B\cap G = E\cup B_1$ with $A_1\cap E
=B_1\cap E =A_1\cap B_1 =\emptyset$. The finite set
$\nu _d(A\cup B)$ is not linearly independent (Remark \ref{a1}), i.e. $h^1(\PP^n,\Ii _{A\cup B}(d)) \ne 0$. Consider the
residual exact sequence of
$G$:
\begin{equation}\label{eqa1}
0 \to \Ii _{Z\setminus Z\cap G}(d-t)\to \Ii _Z(d)\to \Ii _{Z\cap G, G}(d)\to 0
\end{equation}
Since $h^1(\PP^n,\Ii _{Z\setminus Z\cap G}(d-t)) =0$, the long cohomology exact sequence of (\ref{eqa1}) gives $h^1(\PP^n,\Ii _Z(d)) =
h^1(\Ii _{Z\cap G}(d))$, i.e. 
$$\dim \langle \nu _d(Z)\rangle = \dim \langle \nu _d(Z\cap G)\rangle + |Z(A\cup B)\setminus Z\cap G)|,$$i.e. $\nu
_d(Z\setminus Z\cap G)$ is linearly independent and $\langle \nu
_d(Z\setminus Z\cap G)\rangle \cap \langle \nu _d(A\cap G)\rangle =\emptyset$. Since $\nu _d(Z\setminus A\cap G)$ is
linearly independent, while $\nu _d(Z)$ is not linearly independent, we have $Z\cap G\ne \emptyset$. Thus there are
uniquely determined
$q'\in \langle \nu _d(Z\cap G)\rangle$ and $q''\in \langle \nu _d(Z\setminus Z\cap G\rangle$ such that
$q\in \langle \{q',q''\}\rangle$. Since $\nu _d(A)$ minimally spans $q$, $A\cap G\ne \emptyset$ and $A\setminus A\cap G\ne
\emptyset$ we have
$q'\in
\langle \nu _d(A\cap G)\rangle$, $q''\in \langle \nu _d(A\setminus A\cap G)\rangle$, $q\notin \{q',q''\}$ and that $\nu
_d(E\cup A_1)$ minimally spans $q''$. Using $B$ we see that $\nu
_d(E\cup B_1)$ minimally spans $q''$. Since $\nu _d(E\cup A_1\cup B_1)$ is linearly independent, we have
$\langle \nu _d(E\cup A_1)\rangle \cap \langle \nu _d(E\cup B_1)\rangle =\langle \nu _d(E)\rangle$. Since both $\nu _d(E\cup
A_1)$
and $\nu _d(E\cup B_1)$ minimally spans $q''$, we get $A_1=B_1=\emptyset$, i.e. $A\setminus A\cap G =B\setminus B\cap G$.
\end{remark}

\begin{proof}[Proof of part (a) of Theorem \ref{i1}:]
Recall that $\dim \mathbb {C}[z_0,\dots ,z_n]_{\lfloor d/2\rfloor} =\binom{n+\lfloor d/2\rfloor}{n}$. Since $|A|<|S| \le \binom{n+\lfloor d/2\rfloor}{n}$, there is $g\in \CC[z_0,\dots ,z_n]_{\lfloor d/2\rfloor}$ such that $g({p})
=0$ for all $p\in A$. Let $G\subset \PP^n$ be the degree $\lfloor g/2\rfloor$ hypersurface $\{g=0\}$ of $\PP^n$. Since
$A\subset G$, we have $Z\setminus Z\cap G = A\setminus A\cap G$. Thus $Z\setminus Z\cap G$ gives independent conditions to
forms of degree $\lfloor d/2\rfloor$. Thus it gives independent conditions to forms of degree $\lceil d/2\rceil = d-\lfloor
d/2\rfloor$. Since $A\subset G$, Lemma \ref{a2} gives $S\subset G$. Since this is true for all $g\in \CC[z_0,\dots
,z_n]_{\lfloor d/2\rfloor}$ such that $g({p}) =o$ for all $p\in A$, we get that if $g_{|A} =0$ and $g$ has degree $\lfloor
d/2\rfloor$, then $g_{|S}=0$. Since $S$ gives $|S|$ independent linear conditions to $\CC[z_0,\dots ,z_n]_{\lfloor
d/2\rfloor}$, $A$ gives at least $|S|$ linear independent condition to $\CC[z_0,\dots ,z_n]_{\lfloor d/2\rfloor}$,
contradicting the inequality $|A|<|S|$.\end{proof}

\begin{proof}[Proof of Theorem \ref{a3}]
To prove Theorem \ref{a3} we may assume $A\ne S$. Since $|A|=|S|< \binom{n+\lfloor d/2\rfloor}{n}$, there is $g\in \CC[z_0,\dots ,z_n]_{\lfloor d/2\rfloor}$ such that $g_{|A}\equiv 0$. The proof of part (a) of Theorem \ref{i1} gives $g_{|S}\equiv 0$. Thus $H^0(\Ii _A(\lfloor d/2\rfloor )) \supseteq H^0(\Ii _S(\lfloor d/2\rfloor ))$. Since $H^0(\Ii _S(\lfloor d/2\rfloor ))$ has codimension $|A|$ in $\CC[z_0,\dots ,z_n]_{\lfloor d/2\rfloor}$, we get $H^0(\Ii _A(\lfloor d/2\rfloor )) =H^0(\Ii _S(\lfloor d/2\rfloor ))$.\end{proof}

\begin{proof}[Proof of part (b) of Theorem \ref{i1}:] We have $H^0(\Ii _A(\lfloor d/2\rfloor )) =H^0(\Ii _S(\lfloor d/2\rfloor
))$ by Theorem \ref{a3}. To get $A=S$ it is sufficient to prove that for each
$p\in \PP^n\setminus A$ there is $g\in H^0(\Ii _S(\lfloor d/2\rfloor ))$ such that $g({p}) \ne 0$. More precisely it is
sufficient to prove that the sheaf $\Ii _S(\lfloor d/2\rfloor ))$ is generated by its global sections. The assumption that $S$
gives $|S|$ independent conditions to $\CC[z_0,\dots ,z_n]_{\lfloor d/2\rfloor -1}$ is translated in cohomological terms as
$h^1(\PP^n,\Ii _S(\lfloor d/2\rfloor -1))=0$. The sheaf $\Ii _S(\lfloor d/2\rfloor ))$ is generated by its global sections
(and in particular for each $p\in \PP^n\setminus S$ there is $f\in H^0(\Ii _S(\lfloor d/2\rfloor))$ such that $f({p})\ne 0$)
by the Castelnuovo-Mumford's lemma (\cite[Corollary 4.18]{e},
\cite[Theorem 1.8.3]{laz}.
\end{proof}

\section{The case $d=4$}\label{S4}

Set $X:= \nu _d(\PP^n)\subset \PP^r$. For any $q\in \PP^r$ let $\Ss (X,q)$ denote the set of all finite subsets $S\subset X$
evincing the
$X$-rank of
$q$, i.e. the set of all finite subsets $S\subset X$ such that $S$ spans $q$ and $S$ has the minimal cardinality  among all
subsets of
$X$ spanning $q$. By the definition of identifiability with respect to $X$ we have $|\Ss (X,q)|=1$ if and only if $q$ is
identiable.

A finite set $S\subset \PP^n$ is said to be in \emph{linearly general position} (or in LGP, for short) if $\dim \langle
A\rangle = \min \{0,|A|-1\}$ for each $A\subseteq S$. If $|S|\ge n+1$ the set $S$ is in LGP if and only if each $A\subseteq S$
with $|A|=n+1$ spans $\PP^n$.

In this section we take $d=4$ and hence $r=\binom{n+4}{n}-1$.

\begin{theorem}\label{i0}
Fix a finite set $S\supset \PP^n$ such that $|S|=2n+1$ and take $q\in \PP^r$, $r=\binom{n+4}{n}-1$, such that $\nu _4(S)$
minimally spans
$q$. Assume the existence of
$S'\subset S$ such that $|S'|=2n$,
$S'$ is in LGP, but no
$S''$ with
$S'\subsetneq S''\subseteq S$ is in LGP. The point $q$ has rank $2n+1$. Let $e$ be the dimension of a minimal subspace
$N\subset \PP^n$ such that
$|N\cap S|\ge e+2$.  The point $q\in \PP^r$ is identifiable if and only if $e\ge 2$. If $e=1$, then $\dim \Ss (X,q)= 1$.
\end{theorem}

To prove Theorem \ref{i0} we need some elementary observations.

\begin{remark}\label{b1}
Take $A\subset \PP^m$, $m\ge 1$, such that $|A|=m+1$ and $A$ spans $\PP^m$. By induction on $m$ we see that $h^1(\Ii _A(2))
=0$ and that $\Ii _A(2)$ is spanned by its global sections.
\end{remark}

\begin{remark}\label{b2}
Take $A\subset \PP^m$, $m\ge 1$, such that $|A|=m+2$ and $A$ is in LGP. Any two such sets are projectively equivalent. We get
that $h^1(\Ii _A(2))=0$. The sheaf $\Ii _A(2)$ is spanned by its global sections if and only if $m\ge 2$.
\end{remark}

\begin{remark}\label{b3}
Take $A\in \Ss (X,q)$ and any $A'\subsetneq A$, $A'\ne \emptyset$. Set $A'':= A\setminus A'$. In particular $|A|\ge 2$ and
hence
$q\notin X$. Since
$A$ evinces an
$X$-rank, it is linearly independent and
$h^1(\PP^r,\Ii _{A\cup \{q\}}(1)) =1$. Since $A''\subsetneq A$, we have $q\notin \langle A''\rangle$. Thus $\langle
A'\rangle \cap \langle A''\cup \{q\}\rangle$ is a single point, $q'$, and $q'$ is the only element of $\langle A'\rangle$
such that $q\in \langle \{q'\}\cup A''\rangle$. In the same way we see the existence of a single point $q''\in \langle
A''\rangle$ such that $q\in \langle A'\cup \{q''\}\rangle$. We have $q\in \langle \{q',q''\}\rangle$. Since $A\in \Ss (X,q)$,
we have $A'\in \Ss (X,q')$ and $q''\in \Ss (X,q'')$. If we only assume that $A$ minimally spans $q$ the same proof gives the
existence and uniqueness of $q'$ and $q''$ such that $A'$ minimally spans $q'$ and $A''$ minimally spans $q''$.
\end{remark}

\begin{lemma}\label{b5}
Let $H \subset \PP^m$, $m\ge 2$, be a hyperplane. Take a finite set $S\subset \PP^m$ such that $|S\setminus S\cap U|=1$. Take
homogeneous coordinates $z_0,\dots ,z_m$ of $\PP^m$ such that $H = \{z_m=0\}$.. 

\quad (i) If $S\cap H$ imposes independent conditions to $\CC [z_0,\dots ,z_{m-1}]_2$, then $S$ imposes independent conditions
to
$\CC [z_0,\dots ,z_m]_2$.

\quad (ii) If $S\cap H$ is the base locus of $|\Ii _{S\cap H}(2)|$, then $S$ is the base locus of $|\Ii _S(2)|$.
\end{lemma}

\begin{proof}
Set $\{p\}:= S\setminus S\cap H$ and call $\Bb$ the base locus of $|\Ii _{S\cap H}(2)|$. We have the residual exact sequence of
$H$:
\begin{equation}\label{eqb3}
0 \to \Ii _p(1)\to \Ii _S(2) \to \Ii _{S\cap H,H}(2)\to 0
\end{equation}
Since $\{p\}$ imposes independent conditions to $\CC [z_0,\dots ,z_m]_1$, we get part (i) and that the restriction map $\rho :
H^0(\Ii _S(2)) \to H^0(H,\Ii _{S\cap H,H}(2))$ is surjective. Assume that $S\cap H$ is the base locus of $|\Ii _{S\cap H}(2)|$.
Since $\rho$ is surjective, we get $\Bb \cap H =S\cap H$. Fix $o\in \PP^n\setminus H$ such that $o\ne p$. Take a hyperplane
$M\subset \PP^m$ such that $p\in M$ and $o\notin M$. The reducible quadric $H\cup M$ shows that $o\notin \Bb$.
\end{proof}

\begin{proof}[Proof of Theorem \ref{i0}:]
Let $H\subset \PP^n$ be a hyperplane containing $N$ and spanned by points of $S'$. Since $S'$ is in GLP and $|S|=|S'|+1$,
we have
$|S\cap H|+n+1$, $|S'\cap H| =n$, $S\setminus S\cap H = S'\setminus S\cap H$, and $|S'\setminus S'\cap H|=n$. Since $S'$
is in GLP, $S'\setminus S'\cap H$ spans a hyperplane, $M$, and $S'\cap H\cap M=\emptyset$. Set $A:= S'\cap H$ and $B:=
S'\cap M$. Note that $S\subset H\cup M$, $n \le |M\cap S|\le n+1$ and $|S\cap M|=n+1$ if and only if
$S\setminus S'\subset H\cap M$, i.e. if and only if $N\subseteq H\cap M$. Set $\Bb :=
\{p\in \PP^n\mid \dim H^0(\Ii _{S\cup \{p\}}(2)) =\dim H^0(\Ii _S(2))\}$. Since $S\subset H\cup M$ and $H\cup M$ is a quadric
hypersurface, we have
$S\subseteq
\Bb
\subseteq H\cup M$. Consider the
residual exact sequences of $H$ and $M$:
\begin{equation}\label{eqb1}
0 \to \Ii_{S\setminus S\cap H}(1) \to \Ii _S(2)\in \Ii _{S\cap M,M}(2)\to 0
\end{equation}
\begin{equation}\label{eqb2}
0 \to \Ii_{S\setminus S\cap M}(1) \to \Ii _S(2)\in \Ii _{S\cap M,M}(2)\to 0
\end{equation}

Note that $\Bb$ contains the base locus $\Bb _1$ of $\Ii _{S\cap H,H}(2)$ and the base locus $\Bb _2$ of $\Ii _{S\cap
M,M}(2)$.

By Remark
\ref{b1} we have
$h^1(H,\Ii _{S\cap H,H}(2))=h^1(M,\Ii _ {S\cap M}(2))=0$. By the long cohomology exact sequences of (\ref{eqb1}) we get
$h^1(\Ii _S(2)) =0$. Theorem \ref{a3} gives that $q$ has rank
$2n+1$. By the long cohomology exact sequences of (\ref{eqb1}) and
(\ref{eqb2}) the restriction maps $\rho : H^0(\Ii _S(2)) \to H^0(H,\Ii _{S\cap H,H}(2))$ and $\rho ': H^0(\Ii _S(2)) \to
H^0(M,\Ii _{S\cap M,M}(2))$ are surjective. Thus $\Bb =\Bb _1\cup \Bb _2$.  Since $S\cap M$ is linearly independent, we have
$\Bb _2 =S\cap M$. 

\quad (a1) Assume $e\ge 2$. By Remark \ref{b2} $S\cap N$ is the base locus of $\Ii _{S\cap N}(2)$. Applying (if $e<n-1$)
$n-1-e$ times Lemma \ref{b5} we get $\Bb _1 = S\cap H$. Thus $\Bb =S$.

\quad (a2) Assume $e=1$. In this case $\Bb$ contains the line $N$. The proof of Lemma \ref{b5} gives $\Bb _1 = N\cup (S\cap
H)$. Assume the existence of $A\in \Ss (X,q)$ such that $A\ne S$. By Theorem \ref{a3} we have $A\subset N\cap (S\setminus
S\cap N)$. Thus $A = A_1\cup A_2$ with $A_1\subset N$, $A_2\subseteq S\setminus S\cap N$ and $A_1\cap A_2=\emptyset$. We apply
Remark
\ref{b3} with
$A' = N\cap S$ and get
$q'\in \langle \nu _d(S\cap N)\rangle$ and $q''\in \langle \nu _d(S\setminus S\cap N)\rangle$ such that $q \in \langle
\{q',q''\}\rangle$. Since $|S\cap N| =3$, Sylvester's theorem $q'$ has rank $3$ with respect to degree $4$ rational normal
curve $\nu _4(N)$. We get $|A\cap N|\le 3$. Since $|A|=|S|$, we get that each element of $\Ss (X,q)$ is the union of
$S\setminus S\cap N$. By Sylvester's theorem (\cite[\S 1.5]{ik}) we have $\dim \Ss (\nu _d(N),q') =1$.
\end{proof}

\end{document}